\documentclass[11pt]{article}
\usepackage[english,activeacute]{babel}
\usepackage{amsmath,amsfonts,amssymb,amstext,amsthm,amscd,mathrsfs,amsbsy,xypic,centernot}
\usepackage{xypic}

\newtheorem{teo}{Theorem}[section]
\newtheorem{pro}[teo]{Proposition}
\newtheorem{coro}[teo]{Corollary}
\newtheorem{lem}[teo]{Lemma}

\theoremstyle{definition}
\newtheorem{defi}[teo]{Definition}
\newtheorem{exam}[teo]{Example}
\newtheorem{rem}[teo]{Remark}

\DeclareMathOperator{\rank}{rank}
\DeclareMathOperator{\tors}{tors}
\DeclareMathOperator{\Spec}{Spec}

\newcommand{\N}{\mathbb N}

\newcommand{\C}{\mathbb C}

\newcommand{\V}{\mathbf{V}}

\newcommand{\Kx}{k[x_1,\ldots,x_s]}

\newcommand{\Om}{\Omega^{1}}
\newcommand{\Omn}{\Omega^{(n)}_{A/k}}
\newcommand{\Omb}{\Omega^{(n)}_{B/k}}
\newcommand{\Omr}{\Omega^{(n)}_{R/k}}

\newcommand{\Omd}{\Omega^{(2)}_{B/k}}
\newcommand{\dn}{d^{n}_A}
\newcommand{\dnb}{d^{n}_{B}}

\newcommand{\bd}{\bar{d}}

\newcommand{\m}{\mathfrak{m}}
\newcommand{\p}{\mathfrak{p}}
\newcommand{\q}{\mathfrak{q}}
\newcommand{\ida}{\mathfrak{a}}
\newcommand{\idc}{\mathfrak{c}}
\newcommand{\mm}{\mathfrak{M}}

\newcommand{\ml}{\mathcal{L}}

\newcommand{\br}{\bar{R}}
\newcommand{\Rq}{R_{\q}}

\newcommand{\ff}{\bar{F_{\beta}}}
\newcommand{\fb}{F_{\beta}}

\newcommand{\Jac}{Jac_n(f)}
\newcommand{\Jar}{Jac_n(f_1,\ldots,f_r)}

\begin{document}

\title{On the module of differentials of order $n$ of hypersurfaces}
\author{Paul Barajas, Daniel Duarte\footnote{Research supported by CONACyT through the program C\'atedras CONACyT.}}
\maketitle

\begin{abstract}

We give an explicit presentation of the module of differentials of order $n$ of a finitely generated algebra via a higher-order Jacobian matrix.
We use the presentation to study some aspects of this module in the case of hypersurfaces. More precisely, we prove higher-order 
versions of known results relating freness and torsion-freness of the module of differentials with the regularity and normality of 
the hypersurface. We also study its projective dimension.

\end{abstract}


\section*{Introduction}

It is a classical fact that the module of differentials of a finitely generated algebra can be presented using a Jacobian matrix.
Recall that for a $k$-algebra $A$, the module of K\"ahler differentials of $A$ is defined as the quotient $\Om_{A/k}:=I_A/I_A^2$, 
where $I_A\subset A\otimes_kA$ is the kernel of the multiplication map. More generally, a module of K\"ahler differentials
of order $n$ can be defined as $\Omn:=I_A/I_A^{n+1}$ (see, for instance, \cite{G,N,O}).

The first goal of this paper is to give an explicit presentation of $\Omn$ as in case $n=1$. To that end, we first consider some functorial 
properties of the module of differentials proved by Y. Nakai in \cite{N}. Those properties give place to a higher-order Jacobian matrix 
which allows us to obtain the desired presentation.

With this result at hand, we are able to prove higher-order versions of some classical results regarding the module of 
differentials in the case of hypersurfaces. 

It is well-known that, under some conditions, a localization of a finitely generated algebra is regular if and only if its module of differentials
is free. We prove the analogous statement for the module of differentials of order $n$ in the case of hypersurfaces. 

By looking at some basic properties of the higher-order Jacobian matrix of a hypersurface, we also prove that the projective dimension 
of $\Omn$ is less or equal than one (the analogous statement for the module of principal parts was proved by A. Erdogan in \cite{Er}). 
This result, combined with the characterization of regularity in terms of $\Omn$, allows us to conclude that the projective dimension of 
this module is zero if and only if the hypersurface is non-singular.

Another result concerning the module of differentials states that, for local complete intersections, the normality of the ring is equivalent 
to the torsion-freeness of its module of differentials (this was proved by J. Lipman in \cite{L} and, independently, by S. Suzuki in \cite{S}).
We show that, in the case of hypersurfaces, the same result holds for the module of differentials of order $n$. The proof of this fact is 
essentially the same as the one presented in \cite{S}: we simply replace the usual Jacobian matrix by our higher-order Jacobian matrix.
We conclude with some examples that illustrate other features of the torsion submodule of the module of differentials of order $n$.


\section{The module of differentials of order $n$}

All rings we consider are assumed to be commutative and with a unit element.

Let $A$ be a $k$-algebra and $F$ be an $A$-module. Recall that a $k$-linear map $D:A\rightarrow F$ is said to be a derivation if for 
any two elements $x_0,x_1\in A$, the following identity holds:
$$D(x_0x_1)=x_0D(x_1)+x_1D(x_0).$$
A derivation of order $n$ can be defined generalizing the previous identity as follows.

\begin{defi}\cite[Chapter I-1]{N}
A $k$-linear map $D:A\rightarrow F$ is said to be a \textit{derivation of order} $n$ if for any $x_0,\ldots,x_n\in A$, the following 
identity holds:
$$D(x_0\cdots x_n)=\sum_{s=1}^n(-1)^{s-1}\sum_{i_1<\ldots<i_s}x_{i_1}\cdots x_{i_s}
D(x_0\cdots \check{x_{i_1}}\cdots \check x_{i_s}\cdots x_n),$$
where $\check{x_{i_j}}$ means that this element do not appear in the product. Notice that a derivation of order 1 is a usual derivation. 
\end{defi}

The module of derivations of order $n$ can be represented as follows.
Let $I_A$ denote the kernel of the homomorphism $A\otimes_{k}A\rightarrow A$, $a\otimes b\mapsto ab$. Giving structure of $A$-module 
to $A\otimes_{k}A$ by multiplying on the left, we define the $A$-module $$\Omn:=I_A/I_A^{n+1}.$$ Define the map 
$\dn:A\rightarrow\Omn$, $a\mapsto(1\otimes a-a\otimes1)+I_A^{n+1}$. This map is a derivation of order $n$ and its image generates
$\Omn$ as an $A$-module (see \cite[Chapter II-1]{N}).

\begin{defi}
The $A$-module $\Omn$ is called the \textit{module of K\"ahler differentials of order} $n$ of $A$ over $k$. The map $\dn$ is called 
the \textit{canonical derivation} of $A$ in $\Omn$.
\end{defi}

\begin{teo}\cite[Proposition 1.6]{O}
Let $D:A\rightarrow F$ be a derivation of order $n$. Then there exists a unique $A$-linear map $h:\Omn\rightarrow F$ such
that $D=h\circ\dn$.
\end{teo}

If $A=\Kx$, the module $\Omn$ is a free $A$-module with basis $\{(\dn(x))^{\alpha}|\alpha\in\N^s,1\leq|\alpha|\leq n\}$,
where $(\dn(x))^{\alpha}:=(\dn(x_1))^{\alpha_1}\cdots(\dn(x_s))^{\alpha_s}$ (see, for instance, \cite[Chapter II-2]{N}). 
Actually, for an element $f\in A$, 
\begin{equation}\label{Taylor}
\dn(f)=\sum_{\substack{\alpha\in\N^s \\ 1\leq|\alpha|\leq n}}\frac{1}{\alpha!}\frac{\partial^{\alpha}(f)}{\partial x^{\alpha}}(\dn(x))^{\alpha}.
\end{equation}
\begin{rem}
The symbol $\frac{1}{\alpha!}\frac{\partial^{\alpha}(f)}{\partial x^{\alpha}}$ denotes the formal derivative of $f$ if $char(k)=0$
(with the usual multi-index notation). If $char(k)=p>0$, this symbol denotes the coefficient of $t^{\alpha}$ in $f(x+t)-f(t)$ 
(\cite[Chapter II-2]{N}).
\end{rem}


\section{An explicit presentation of $\Omb$ for a finitely generated $k$-algebra.}

In this section we give an explicit presentation of the module of differentials of order $n$ for a finitely generated $k$-algebra. 
It is well-known that for $n=1$, the Jacobian matrix presents the module of differentials. We define a higher-order Jacobian
matrix that, combined with a fundamental exact sequence of the module of differentials, gives the corresponding presentation
for $n\geq1$.

Let $A$ be a $k$-algebra, $J\subset A$ an ideal, and $B=A/J$. Consider the $k$-linear map
\begin{align}
\rho:J/&J^{n+1}\rightarrow\Omn\otimes_A B\notag\\
&\bar{c}\mapsto\dn(c)\otimes1\notag
\end{align}
Denote as $\langle Im(\rho)\rangle_B$ the $B$-module generated by the image of $\rho$ in $\Omn\otimes_A B$.

\begin{teo}\cite[Theorem II-14]{N}\label{sec ex seq}
With the previous notation, the following sequence of $B$-modules is exact
\begin{equation}\label{second seq}
\xymatrix{0\ar[r]&\langle Im(\rho)\rangle_B\ar[r]&\Omn\otimes_A B\ar[r]^{\tau}&\Omb\ar[r]&0},\notag
\end{equation}
where $\tau(\dn(a)\otimes b)=b\dnb(\bar{a})$. In particular,
$$\Omb\cong\big(\Omn\otimes_A B\big)/\langle Im(\rho)\rangle_B.$$
\end{teo}

Our first goal is to describe explicitly $\langle Im(\rho)\rangle_B$ if $A$ is a polynomial ring and generators of the ideal $J$ 
are known. So, for the rest of this section, $A=\Kx$, $J=\langle f_1,\ldots,f_r \rangle$, and $B=A/J$. First we describe $Im(\rho)$.

\begin{lem}\label{im rho}
With the previous notation,
$$Im(\rho)=\big\{\sum_{i=1}^r \big(\dn(g_i)\dn(f_i)\otimes 1+g_i(\dn(f_i)\otimes 1)\big)|g_1,\ldots,g_r\in A\big\}.$$
\end{lem}
\begin{proof}
For any $a,b\in A$, the following identity holds in $A\otimes_k A$:
$$1\otimes ab-ab\otimes1=(1\otimes a-a\otimes1)(1\otimes b-b\otimes1)+a(1\otimes b-b\otimes1)+b(1\otimes a-a\otimes1).$$
In particular, for any $a,b\in A$,
\begin{equation}\label{identity}\notag
\dn(ab)=\dn(a)\dn(b)+a\dn(b)+b\dn(a).
\end{equation}
Now we proceed to prove the lemma. Let $h=\sum_{i=1}^r g_if_i\in J$. Then
\begin{align}
\rho(\bar{h})=&\sum_i \rho(\overline{g_if_i})=\sum_i \dn(g_if_i)\otimes1\notag\\
=&\sum_i \dn(g_i)\dn(f_i)\otimes1+g_i(\dn(f_i)\otimes1)+f_i(\dn(g_i)\otimes1).\notag
\end{align}
To show that $\rho(\bar{h})$ is in the right-hand side of the claimed equality, we only need to show that $f_i(\dn(g_i)\otimes1)=0$ 
for each $i$. Since the tensor is over $A$ we have:
$$f_i(\dn(g_i)\otimes1)=\dn(g_i)\otimes\bar{f_i}=\dn(g_i)\otimes0=0.$$
The other inclusion is shown just by reversing the previous argument.
\end{proof}

Now we want to give a more precise description of $\langle Im(\rho)\rangle_B$. We need to introduce some notation. 
For $\beta\in\N^s$ such that $0\leq|\beta|\leq n-1$, denote
\begin{equation}\label{efes}
F_{\beta}^i:=(\dn(x))^{\beta}\dn(f_i).
\end{equation}

\begin{exam}\label{efes cusp}
Let $f=x_1^3-x_2^2\in A=k[x_1,x_2]$. Let $n=2$. Then by (\ref{Taylor}) and recalling that $\Omega^{(2)}_{A/k}=I_A/I_A^3$:
\begin{align}
F_{(0,0)}&=(d^2_A(x))^{(0,0)}d^2_A(f)=3x_1^2d^2_A(x_1)-2x_2d^2_A(x_2)+3x_1d^2_A(x_1)^2-d^2_A(x_2)^2,\notag\\
F_{(1,0)}&=(d^2_A(x))^{(1,0)}d^2_A(f)=3x_1^2d^2_A(x_1)^2-2x_2d^2_A(x_1)d^2_A(x_2),\notag\\
F_{(0,1)}&=(d^2_A(x))^{(0,1)}d^2_A(f)=3x_1^2d^2_A(x_1)d^2_A(x_2)-2yd^2_A(x_2)^2.\notag
\end{align}
\end{exam}

Let us describe more precisely the elements $F^i_{\beta}$. Fix $\beta\in\N^s$ such that $0\leq|\beta|\leq n-1$. 
By definition and (\ref{Taylor}):
$$F_{\beta}^i:=(\dn(x))^{\beta}\dn(f_i)=\sum_{\substack{\gamma\in\N^s\\1\leq|\gamma|\leq n}}
\frac{1}{\gamma!}\frac{\partial^{\gamma}(f_i)}{\partial x^{\gamma}}(\dn(x))^{\gamma+\beta}.$$
By making $\alpha=\gamma+\beta$ and recalling that $\Omn=I_A/I_A^{n+1}$ we obtain
\begin{align}
F_{\beta}^i=&\sum_{\substack{\alpha\in\N^s \\ 1+|\beta|\leq|\alpha|\leq n+|\beta|,\\ \beta\leq\alpha,\beta\neq\alpha}}
\frac{1}{(\alpha-\beta)!}\frac{\partial^{\alpha-\beta}(f_i)}{\partial x^{\alpha-\beta}}(\dn(x))^{\alpha}\notag\\
=&\sum_{\substack{\alpha\in\N^s \\ 1+|\beta|\leq|\alpha|\leq n,\\ \beta\leq\alpha,\beta\neq\alpha}}
\frac{1}{(\alpha-\beta)!}\frac{\partial^{\alpha-\beta}(f_i)}{\partial x^{\alpha-\beta}}(\dn(x))^{\alpha}\notag
\end{align}
(here $\beta\leq\alpha$ means $\beta_i\leq\alpha_i$ for each $i\in\{1,\ldots,s\}$). To simplify the previous expression, 
we define
\[\frac{1}{(\alpha-\beta)!}\frac{\partial^{\alpha-\beta}(f_i)}{\partial x^{\alpha-\beta}}=0, \mbox{ whenever }
\left\{
\begin{array}{rll}\label{star}
&|\alpha|<1+|\beta|,\\
&\alpha_i<\beta_i \mbox{ for some }i, \\ \tag{*}
&\alpha=\beta.
\end{array}
\right.
\]
With this convention, the expression for $F^i_{\beta}$ is written more simply as:
\begin{equation}\label{efes2}
F_{\beta}^i=\sum_{\substack{\alpha\in\N^s\\1\leq|\alpha|\leq n}}
\frac{1}{(\alpha-\beta)!}\frac{\partial^{\alpha-\beta}(f_i)}{\partial x^{\alpha-\beta}}(\dn(x))^{\alpha}.
\end{equation}

\begin{lem}\label{F betas}
Let $\ml\subset\Omn\otimes_A B$ be the $B$-submodule generated by the set 
$\{F_{\beta}^i\otimes1|0\leq|\beta|\leq n-1,i=1,\ldots,r\}$. Then $\ml=\langle Im(\rho)\rangle_B$.
\end{lem}
\begin{proof}
Let us start with the inclusion $\langle Im(\rho)\rangle_B\subset \ml$. It is enough to show that $Im(\rho)\subset \ml$.
First notice that, for $g\in A$ and $i\in\{1,\ldots,r\}$,
$$g(\dn(f_i)\otimes1)=g(F^i_{(0,\ldots,0)}\otimes1)=F^i_{(0,\ldots,0)}\otimes\bar{g}=\bar{g}(F^i_{(0,\ldots,0)}\otimes1)\in\ml$$ 
(here we are using the $A$-module structure of $\Omn\otimes_A B$ as well as its $B$-module structure).
Similarly, using (\ref{Taylor}) and recalling that $\Omn=I_A/I_A^{n+1}$,
\begin{align}
\dn(g)\dn(f_i)&=\Big(\sum_{1\leq|\alpha|\leq n}\frac{1}{\alpha!}
\frac{\partial^{\alpha}g}{\partial x^{\alpha}}(\dn(x))^{\alpha}\Big)\dn(f_i)\notag\\
&=\sum_{1\leq|\alpha|< n}\frac{1}{\alpha!}\frac{\partial^{\alpha}g}{\partial x^{\alpha}}(\dn(x))^{\alpha}\dn(f_i)+
\sum_{|\alpha|=n}\frac{1}{\alpha!}\frac{\partial^{\alpha}g}{\partial x^{\alpha}}(\dn(x))^{\alpha}\dn(f_i)\notag\\
&=\sum_{1\leq|\alpha|< n}\frac{1}{\alpha!}\frac{\partial^{\alpha}g}{\partial x^{\alpha}}(\dn(x))^{\alpha}\dn(f_i)
=\sum_{1\leq|\alpha|< n}\frac{1}{\alpha!}\frac{\partial^{\alpha}g}{\partial x^{\alpha}}F_{\alpha}^i.\notag
\end{align}
Therefore $\dn(g)\dn(f_i)\otimes1\in\ml$. By lemma \ref{im rho}, $Im(\rho)\subset\ml$. 

To prove the other inclusion, first notice that 
$F^i_{(0,\ldots,0)}\otimes1=\dn(f_i)\otimes1=\rho(\bar{f_i})\in Im(\rho)$. On the other hand, according to \cite[Proposition II-2]{N},
the set $\{\dn(x^{\alpha})|1\leq|\alpha|\leq n\}$ is also a basis for $\Omn$. Fix $\beta\in\N^s$ such that $1\leq|\beta|\leq n-1$.
Then we can write $(\dn(x))^{\beta}=\sum_{\alpha}g_{\alpha}\dn(x^{\alpha})$, for some $g_{\alpha}\in A$. 
We want to show that $F^i_{\beta}\otimes1=(\dn(x))^{\beta}\dn(f_i)\otimes1=
\sum_{\alpha}g_{\alpha}\dn(x^{\alpha})\dn(f_i)\otimes1\in\langle Im(\rho) \rangle_B$.

By lemma \ref{im rho} we know that $\dn(x^{\alpha})\dn(f_i)\otimes1+x^{\alpha}\dn(f_i)\otimes1\in Im(\rho)$.
Thus, as before, $g_{\alpha}(\dn(x^{\alpha})\dn(f_i)\otimes1+x^{\alpha}\dn(f_i)\otimes1)\in\langle Im(\rho)\rangle_B$.
Similarly, $\dn(f_i)\otimes1\in Im(\rho)$ implies
$g_{\alpha}(x^{\alpha}\dn(f_i)\otimes1)\in\langle Im(\rho)\rangle_B$. It follows that
$g_{\alpha}\dn(x^{\alpha})\dn(f_i)\otimes1\in\langle Im(\rho) \rangle_B$ and so $F^i_{\beta}\otimes1\in\langle Im(\rho) \rangle_B$.
\end{proof}

\begin{pro}\label{omb}
$$\Omb\cong\bigoplus_{1\leq|\alpha|\leq n}B(\dnb(x))^{\alpha}/\sum_{\substack{i=1,\ldots,r,\\0\leq|\beta|\leq n-1 }} BF_{\beta}^i.$$
\end{pro}
\begin{proof}
By theorem \ref{sec ex seq}, 
$$\Omb\cong\Omn\otimes_A B/\langle Im(\rho) \rangle_B.$$
We also know that $\Omn\cong\oplus_{1\leq|\alpha|\leq n}A(\dn(x))^{\alpha}$. The result follows from basic properties of the tensor
product and lemma \ref{F betas}.
\end{proof}

Now we want to use this result to give a presentation of the module $\Omb$ using a higher-order Jacobian matrix, which we define below.

Let $N=\binom{s+n}{s}$ and $M=\binom{s+n-1}{s}$. Denote as $\tilde{F_{\beta}^i}\in B^{N-1}$ the vector whose entries are the 
coefficients of $F_{\beta}^i$ in (\ref{efes2}). Notice that there are $rM$ such vectors.

\begin{defi}
Denote as $\Jar$ the matrix whose rows are the $rM$ vectors $\tilde{F_{\beta}^i}$. In particular, $\Jar$ is a $(rM\times N-1)$-matrix. 
We call $\Jar$ the \textit{Jacobian matrix of order} $n$ \textit{of} $J$ or the \textit{higher-order Jacobian matrix of} $J$.
\end{defi}

\begin{rem}
For $J=\langle f \rangle$, the matrix $\Jac$ was introduced in \cite{D} in relation to the higher Nash blowup of a hypersurface. 
See also \cite{DG} for a higher-order Jacobian matrix that represents a higher-order derivative of a morphism.
\end{rem}

The next theorem generalizes the relations found in \cite[Section 4.2]{D}. 

\begin{teo}\label{present}
The $B$-module $\Omb$ has the following presentation:
\begin{equation}\label{presentation}
\xymatrix{B^{rM}\ar[r]^{T}&B^{N-1}\ar[r]&\Omb\ar[r]&0},\notag
\end{equation}
where $T$ denotes the transpose of $\Jar$.
\end{teo}
\begin{proof}
The result follows from proposition \ref{omb} and the construction of the matrix $\Jar$.
\end{proof}

\begin{exam}
Let $f=x_1^3-x_2^2\in k[x_1,x_2]$, $B=k[x_1,x_2]/\langle f \rangle$ and $n=2$. By definition and example \ref{efes cusp}:
\begin{align}
Jac_2(f)=
\begin{pmatrix}
3x_1^2&		-2x_2& 	3x_1&  		 0&     		-1\\
0&  	 		0&    	3x_1^2&   	-2x_2&      	0\\
0&       		0& 		0&      		3x_1^2&   	-2x_2 \notag
\end{pmatrix}.
\end{align}
Notice that this matrix have rank 3 (independently of the characteristic of $k$). Denoting by $T$ the transpose of this matrix, 
theorem \ref{present} gives the following free resolution of $\Omd$:
$$\xymatrix{0\ar[r]&B^{3}\ar[r]^{T}&B^{5}\ar[r]&\Omd\ar[r]&0}.$$
\end{exam}

We conclude this section with two results on the higher-order Jacobian matrix that we will need in subsequent sections. 
First we generalize the usual Jacobian criterion for hypersurfaces (see, for instance, \cite[Theorem 13.10]{K}, see 
also \cite[Theorem 2.1]{D} for a particular case). We need the following lemma.

\begin{lem}\label{echelon}
Let $f\in A$. Then, by ordering increasingly the columns and rows of $\Jac$ using graded lexicographical order on $\N^s$, 
$\Jac$ is in row echelon form and has $\frac{\partial f}{\partial x_1}$ as pivots.
\end{lem}
\begin{proof}
By (\ref{efes2}), for each $\beta$ such that $0\leq|\beta|\leq n-1$, the coefficient of $\dn(x)^{\beta+(1,0,\ldots,0)}$ in $\fb$ is 
$\frac{\partial f}{\partial x_1}$. To prove the lemma we need to show that if $\alpha<_{grlex}\beta+(1,0,\ldots,0)$ then the
$(\beta,\alpha)$-entry of $\Jac$ is zero. It was proved in \cite[Lemma 1.1 (iv)]{D} that this is indeed the case if we assume 
that $\alpha\neq\beta$. Finally, according to (\ref{star}) and (\ref{efes2}), the $(\beta,\beta)$-entry of $\Jac$ is also zero.
\end{proof}

\begin{pro}\label{jacobian criterion}
Let $k$ be a perfect field. Let $f\in A$ be an irreducible polynomial. Let $J=\langle f \rangle\subset A$ and $\bar{\q}\in\Spec(A/J)$. 
Then $\big(A/J\big)_{\bar{\q}}$ is a regular local ring if and only if $\rank(\Jac\mod\q)=M$.
\end{pro}
\begin{proof}
Suppose that $\big(A/J\big)_{\bar{\q}}$ is regular. Then $\rank(Jac(f)\mod\q)=1$, by the usual Jacobian criterion. Assume that
$\frac{\partial f}{\partial x_1}\mod\q\neq0$. By lemma \ref{echelon}, the matrix $\Jac$ is in row echelon form and has 
$\frac{\partial f}{\partial x_1}$ as pivots. Therefore, the rows of $\Jac\mod\q$ are linearly independent, i.e., the rank
of $\Jac\mod\q$ is $M$.

Now assume that $\rank(\Jac\mod\q)=M$. In particular, the row of $\Jac\mod\q$ corresponding to $\beta=(0,\ldots,0,n-1)$ contains 
non-zero entries. On the other hand, by construction, this particular row contains only first partial derivatives. It follows that 
$\frac{\partial f}{\partial x_i}\mod\q\neq0$ for some $i$. Then $\big(A/J\big)_{\bar{\q}}$ is a regular local ring by the usual Jacobian
criterion.
\end{proof}

\begin{pro}\label{lin indep}
Let $f\in A$ be an irreducible non-constant polynomial. The set $\{\fb|0\leq|\beta|\leq n-1\}\subset\Omn$ is 
$A$-linearly independent. 
\end{pro}
\begin{proof}
Since $f$ is irreducible then $\frac{\partial f}{\partial x_i}\neq0$ for some $i$. Assume that $i=1$. By definition, the entries of $\Jac$ 
are the coefficients of $\fb$. By lemma \ref{echelon}, this matrix is in row echelon form and has $\frac{\partial f}{\partial x_1}$ 
as pivots. The result follows.
\end{proof}


\section{Projective dimension of $\Omb$ for hypersurfaces}

In \cite{Er}, A. Erdogan studies the module of principal parts using similar ideas to the ones we developed so far for $\Omb$. 
His main theorem \cite[Theorem]{Er} states that $pd(B\otimes_k B/I_B^{n+1})\leq1$, where $B=\Kx/\langle f \rangle$, 
$char(k)=0$, and $pd(\cdot)$ denotes the projective dimension.

In this section we refine the corresponding statement for the module of differentials of order $n$ using the higher-order Jacobian matrix. 
We show that, for hypersurfaces,  $pd(\Omb)\leq1$, and it is exactly one if and only if the hypersurface is singular. In addition, our 
statement is valid also in positive characteristic.

First we need to generalize, for hypersurfaces and the module of differentials of order $n$, the well-known criterion of regularity of a ring 
in terms of $\Om$. For the rest of this section we use the following notation: 
\begin{itemize}
\item $k$ a perfect field, $A=\Kx$, $f\in A$ irreducible, $B=A/\langle f \rangle$.
\item $\mm\in\Spec_{max}(B)$, $R=B_{\mm}$, $\m$ the maximal ideal of $R$. 
\item We assume that $k$ is the residue field of $R$.
\end{itemize}

\begin{teo}\label{regular}
$R$ is a regular ring if and only if $\Omr$ is a free $R$-module of rank $L-1$, where $L=\binom{s-1+n}{s-1}$.
\end{teo}
\begin{proof}
The ``only if" part is well-known and holds in full generality: if $R$ is a regular ring then $\Om_R$ is free of rank $\dim R$; 
using the exact sequences
$$0\rightarrow I_R^{n}/I_R^{n+1}\rightarrow I_R/I_R^{n+1}\rightarrow I_R/I_R^{n}\rightarrow0,$$
the proof follows by induction. 

Now suppose that $\Omr$ is a free $R$-module of rank $L-1$. As in the case of $\Om_R$, it is known that
$\Omr\otimes_R k\cong\m/\m^{n+1}$ (this fact follows from theorem \ref{sec ex seq}). Thus $\m/\m^{n+1}\cong k^{L-1}$, 
and so $\dim_k\m/\m^{n+1}=L-1$. To conclude, we need to show that this fact implies $\dim_k\m/\m^2=s-1$, i.e.,
$\dim_k\m/\m^2=\dim R$. This last statement is proved in the following lemma.
\end{proof}

If $k=\C$, it was proved in \cite[Corollary 2.8]{D} that $\dim_k\m/\m^{n+1}=L-1$ implies $\dim_k\m/\m^2=s-1$. The proof given 
there is essentially combinatorial and so it also holds for arbitrary fields. For convenience of the reader, we reproduce here the 
key arguments of that proof for an arbitrary field.

\begin{lem}
If $\dim_k\m/\m^{n+1}=L-1$ then $\dim_k\m/\m^2=s-1$.
\end{lem}
\begin{proof}
We divide the proof into several steps.\\
\textit{Step 1:} Let $\ida=\langle x^{\gamma}\rangle\subset A$, where
$x^{\gamma}=x_1^{\gamma_1}\cdots x_s^{\gamma_s}$ and assume $\gamma_1>0$. Let 
$\mm=\langle x_1,\ldots,x_s\rangle\subset A/\ida$. We claim that
\begin{equation}\label{symm}
\dim_k \mm^{n}/\mm^{n+1}\geq\binom{s-2+n}{s-2}.
\end{equation}
Let $l=|\gamma|$. If $n<l$,  $\dim_k \mm^{n}/\mm^{n+1}=\binom{s-1+n}{s-1}$ and the statement holds. Now suppose that 
$n=l+j$, $j\geq0$. Let $C=\{x^{\alpha}\notin\ida||\alpha|=l+j\}$. The image of $C$ generates $\mm^{l+j}/\mm^{l+j+1}$ and it is
linearly independent since $\ida$ is a monomial ideal. Let $L_j:=|\{x^{\alpha}\in\ida||\alpha|=l+j\}|$. Notice that 
$|C|=\binom{s-1+l+j}{s-1}-L_j.$ On the other hand, the set of monomials $x^{\alpha}$ such that $|\alpha|=l+j$ and $\alpha_1>0$
has cardinality $\binom{s-2+l+j}{s-1}$. Therefore (\ref{symm}) holds also for these values of $n$ since
$$\binom{s-1+l+j}{s-1}-L_j\geq\binom{s-1+l+j}{s-1}-\binom{s-2+l+j}{s-1}=\binom{s-2+l+j}{s-2}.$$
\textit{Step 2:} Let $f\in A$, $\mm=\langle x_1,\ldots,x_s\rangle\subset A/\langle f \rangle$. We claim that 
\begin{equation}\label{symf}
\dim_k \mm^{n}/\mm^{n+1}\geq\binom{s-2+n}{s-2}.
\end{equation}
This follows immediately from (\ref{symm}) and the fact that the Hilbert functions of $A/\langle f \rangle$, $A/\langle f_0 \rangle$
and $A/\langle in_>(f_0) \rangle$ coincide, where $f_0$ is the homogeneous component of $f$ of lowest degree and $>$ is any
monomial order on $A$ (see \cite[Theorem 15.26, Section 15.10.3]{E}).\\
\textit{Step 3:} Now we are ready to prove the lemma. After a change of coordinates, we can assume that 
$\mm=\langle x_1,\ldots,x_s \rangle\subset B$. It is enough to prove the lemma for $\mm/\mm^{n+1}$ since
$\mm/\mm^{n+1}\cong\m/\m^{n+1}$ (see, for instance, \cite[Lemma 2.4]{D}). Consider the exact sequence
$$0\rightarrow \mm^{n}/\mm^{n+1}\rightarrow \mm/\mm^{n+1}\rightarrow \mm/\mm^{n}\rightarrow0.$$
Suppose that $\dim_k\mm/\mm^{2}>s-1$. Using (\ref{symf}) and this exact sequence, an inductive argument 
shows that $\dim_k\mm/\mm^{n+1}>\binom{s-1+n}{s-1}-1=L-1$.
\end{proof}

\begin{rem}\label{principal}
For a $k$-algebra $A$, a closely related notion to the module of differentials $\Omn$ is the module of principal parts, 
which is defined as $A\otimes_k A/I_A^{n+1}$. These two modules are related by the following isomorphism:
\begin{equation}\label{diff princ}
A\otimes_k A/I_A^{n+1}\cong\Omn\oplus A.
\end{equation}
\end{rem}

\begin{coro}
Let $(R,\m)$ be as before. Then $R\otimes_k R/I_R^{n+1}$ is free of rank $L$ if and only if $R$ is regular.
\end{coro}
\begin{proof}
If $R$ is regular, then $\Omr$ is free of rank $L-1$. By the previous remark, $R\otimes_k R/I_R^{n+1}$ is free of rank $L$.

Now suppose that $R\otimes_k R/I_R^{n+1}$ is free of rank $L$. By (\ref{diff princ}), $\Omr$ is a projective $R$-module. In particular,
$(\Omr)_{\m}$ is a free $R_{\m}$-module. Again by (\ref{diff princ}), it follows that the rank of this free module is $L-1$. On the other
hand, by \cite[Theorem II-9]{N}, $(\Omr)_{\m}\cong\Omega^{(n)}_{R_{\m}/k}$. Since $R$ is local with maximal ideal $\m$ it follows 
that $R_{\m}\cong R$. Therefore, $\Omr$ is a free $R$-module of rank $L-1$. By theorem \ref{regular}, $R$ is regular.
\end{proof}

Now we prove the result regarding the projective dimension of $\Omb$.

\begin{teo}\label{pd}
Let $B$ be as before. Then $pd(\Omb)\leq1$. In addition, $pd(\Omb)=0$ if and only if $\V(f)$ is a non-singular hypersurface.
\end{teo}
\begin{proof}
We can assume that $\frac{\partial f}{\partial x_1}\neq0$. By lemma \ref{echelon}, $\Jac$ is in row echelon form with 
$\frac{\partial f}{\partial x_1}$ as pivots. Therefore, $\Jac$ has rank $M$. In particular, the sequence of corollary \ref{present} 
gives place to the exact sequence
$$\xymatrix{0\ar[r]&B^{M}\ar[r]^{T}&B^{N-1}\ar[r]&\Omb\ar[r]&0}.$$
Therefore $pd(\Omb)\leq1$. For the second part of the theorem recall the following facts ($K$ denotes the quotient field of $B$) :
\begin{enumerate}
\item The $B$-module $\Omb$ is projective if and only if $\big(\Omb\big)_{\mm}$ is a free $B_{\mm}$-module for every 
maximal ideal $\mm\subset B$. In addition, in this case $\rank(\big(\Omb\big)_{\mm})=\dim_K \Omb\otimes_B K$, for 
every maximal ideal $\mm\subset B$.
\item $\dim_K \Omb\otimes_B K=L-1$ (this is just the local version of known results on the sheaf of principal parts, see 
\cite[Paragraph 16.3.7]{G} and \cite[Section 4]{LT}).
\item $\big(\Omb\big)_{\mm}\cong\Omega^{(n)}_{B_{\mm}/k}$ (\cite[Theorem II-9]{N}).
\end{enumerate}
These facts, together with theorem \ref{regular}, concludes the proof of the theorem.
\end{proof}


\section{Torsion submodule of $\Omega^{(n)}_{B_{\bar{\p}}/k}$ for hypersurfaces}

It was proved by J. Lipman (\cite[Proposition 8.1]{L}) and, independently, by S. Suzuki (\cite[Theorem]{S}), that for a (local) 
complete intersection $R$, the module $\Om_R$ has no torsion if and only if $R$ is normal. In this section we show that, using 
our higher Jacobian matrix, the strategy followed by Suzuki can be directly generalized to $\Omr$, whenever $R$ is a hypersurface. 
We conclude with some examples that illustrate further properties of the torsion submodule of the module of differentials of order $n$.


\subsection{A characterization of torsion-freness of $\Omega^{(n)}_{B_{\bar{\p}}/k}$ for hypersurfaces}

We will use the following notation throughout this section:
\begin{itemize}
\item $k$ a perfect field, $A=\Kx$, $f\in A$ irreducible.
\item $B=A/\langle f \rangle$, $K$ the quotient field of $B$.
\item $\bd:A\rightarrow \Omn\otimes_A K$, $a\mapsto \dn(a)\otimes 1$.
\item $R=B_{\bar{\p}}$, where $\bar{\p}\in\Spec(B)$.
\end{itemize}

\begin{rem}\label{free omr}
$\Omn$ is a free $A$-module with basis $\{\dn(x)^{\alpha}|\alpha\in\N^s,1\leq|\alpha|\leq n\}$. Thus, the $R$-module
$\Omn\otimes_A R$ is a free $R$-module with basis $\{\bd(x)^{\alpha}|\alpha\in\N^s,1\leq|\alpha|\leq n\}$.
\end{rem}

\begin{rem}\label{omr quot}
Since the formation of the module of differentials commutes with localization, proposition \ref{omb}
implies
$$\Omr\cong\Big(\Omn\otimes_A R\Big)/\sum_{0\leq|\beta|\leq n-1}R\ff,$$
where $\ff:=F_{\beta}\otimes1$.
\end{rem}

As we mentioned before, the following theorem was proved in \cite{L,S} for $n=1$. To prove the corresponding result for $n\geq1$,
we adapt the strategy followed in \cite{S}. 

\begin{teo}\label{tors}
Let $n\geq1$. Then $R$ is normal if and only if $\tors(\Omr)=0$, where $\tors(\cdot)$ denotes the torsion submodule.
\end{teo}
\begin{proof}
The result is proved in propositions \ref{pro1} and \ref{pro2}.
\end{proof}

\begin{pro}\label{pro1}
If $R$ is normal then $\tors(\Omr)=0$.
\end{pro}
\begin{proof}
Let $z=\big[\sum_{1\leq|\alpha|\leq n}a_{\alpha}\bd(x)^{\alpha}\big]\in\tors(\Omr)$, where $a_{\alpha}\in R$. 
By remark \ref{omr quot}, there exists $h\in R$, $h\neq0$,  $l_{\beta}\in R$, $0\leq|\beta|\leq n-1$ such that
\begin{equation}\label{rel1}
h\big(\sum_{1\leq|\alpha|\leq n}a_{\alpha}\bd(x)^{\alpha}\big)=\sum_{0\leq|\beta|\leq n-1}l_{\beta}\ff.
\end{equation}
From this equation, to show that $z=0$ it is enough to show that $l_{\beta}/h\in R$, for every $\beta$. Since $R$ is normal, $R$ 
is the intersection of all localizations $\Rq$, where $\q\in\Spec(R)$ and $ht(\q)=1$. Therefore, it is enough to show that 
$l_{\beta}/h\in\Rq$, for every such $\q$.

According to (\ref{efes2}), for each $\beta$,
\begin{equation}\label{rel1.1}
\ff=\sum_{\substack{\alpha\in\N^s \\ 1\leq|\alpha|\leq n}}\frac{1}{(\alpha-\beta)!}\frac{\partial^{\alpha-\beta}(f)}
{\partial x^{\alpha-\beta}}\bd(x)^{\alpha}.
\end{equation}
Since the $\bd(x)^{\alpha}$'s form a basis of $\Omn\otimes_A R$, equations (\ref{rel1}) and (\ref{rel1.1}) give place to the 
following equations, for each $\alpha$,
\begin{equation}\label{rel2}
\sum_{0\leq|\beta|\leq n-1}l_{\beta}\frac{1}{(\alpha-\beta)!}
\frac{\partial^{\alpha-\beta}(f)}{\partial x^{\alpha-\beta}}=ha_{\alpha}.
\end{equation}
In other words, denoting $\Jac^t$ the transpose of $\Jac$, we have the relation
\begin{equation}\label{rel3}
\Jac^t\cdot(l_{\beta})_{0\leq|\beta|\leq n-1}=(ha_{\alpha})_{1\leq|\alpha|\leq n}.
\end{equation}

On the other hand, since $ht(\q)=1$ and $R$ is normal, $\Rq$ is a regular ring. By proposition \ref{jacobian criterion}, 
$\rank(\Jac\mod\q)=M$. In particular, there exists a $M\times M$-submatrix $L$ of $\Jac$ such that $\det(L)\mod\q\neq0$,
i.e., $\det(L)\in R\setminus\q$. Considering the adjoint matrix of $L$ and using (\ref{rel3}), we can find $r_{\beta}\in R$, for each
$0\leq|\beta|\leq n-1$, such that
$$\frac{l_{\beta}}{h}=\frac{r_{\beta}}{\det(L)}\in\Rq.$$
\end{proof}

For the ``if" part of the theorem we need the following lemma which appears in the middle of the proof of \cite[Theorem]{S}.

\begin{lem}\label{ass}
Let $R$ be as before and suppose that $R$ is not normal. Let $\br$ be its normalization.
Let $\idc=\{r\in R|r\br\subset R\}$ be the conductor ideal, $\q\subset R$ an associated prime of $\idc$, and 
$\idc'=(\idc:\q)=\{h\in R|h\q\subset\idc\}$. Then
\begin{itemize}
\item[(i)] $\Rq$ is not normal.
\item[(ii)] There exists $g\in\idc'\br$ such that $g\notin\Rq$.
\end{itemize}
\end{lem}

\begin{pro}\label{pro2}
If $\tors(\Omr)=0$ then $R$ is normal.
\end{pro}
\begin{proof}
The proof is by contradiction. Suppose that $R$ is not normal. In what follows we use the notation of lemma \ref{ass}.
By that lemma, $\Rq$ is not normal, so it is not regular. By proposition \ref{jacobian criterion}, $\rank(\Jac\mod\q)<M$. Therefore, 
there exists a nonzero vector $v=(c_{\beta}/d_{\beta})_{0\leq|\beta|\leq n-1}\in Q(R/\q)^M$ such that $v\in\ker\Jac^t$ 
($Q(R/\q)$ denotes the field of fractions of $R/\q$). Therefore, for each $\gamma$, $1\leq|\gamma|\leq n$:
\begin{equation}\label{eq1}
\sum_{\substack{0\leq|\beta|\leq n-1 \\ \beta\leq\gamma, \beta\neq\gamma}}\frac{1}{(\gamma-\beta)!}
\frac{\partial^{\gamma-\beta}(f)}{\partial x^{\gamma-\beta}}\frac{c_{\beta}}{d_{\beta}}=0.
\end{equation}
Let $d:=\prod_{0\leq|\beta|\leq n-1}d_{\beta}$ and $a_{\beta}:=dc_{\beta}/d_{\beta}\in R$. Multiplying equations (\ref{eq1}) by $d$
we obtain the following equalities in $R/\q$ for each $\gamma$:
\begin{equation}\label{eq2}
\sum_{\substack{0\leq|\beta|\leq n-1 \\ \beta\leq\gamma, \beta\neq\gamma}}
\frac{1}{(\gamma-\beta)!}\frac{\partial^{\gamma-\beta}(f)}{\partial x^{\gamma-\beta}}a_{\beta}=0.\notag
\end{equation}
In other words, the left-side of this equation is an element of $\q$. Using this fact we obtain (by making $\gamma=\alpha+\beta$ 
in the fourth row):
\begin{align}
\sum_{0\leq|\beta|\leq n-1}a_{\beta}\ff&=\sum_{0\leq|\beta|\leq n-1}a_{\beta}
\Big(\sum_{1\leq|\alpha|\leq n}\frac{1}{\alpha!}\frac{\partial^{\alpha}(f)}{\partial x^{\alpha}}\bd(x)^{\alpha+\beta}
\Big)\notag\\
&=\sum_{0\leq|\beta|\leq n-1}\sum_{1\leq|\alpha|\leq n}
\frac{1}{\alpha!}\frac{\partial^{\alpha}(f)}{\partial x^{\alpha}}\bd(x)^{\alpha+\beta}a_{\beta}\notag\\
&=\sum_{1\leq|\alpha|\leq n}\sum_{0\leq|\beta|\leq n-1}
\frac{1}{\alpha!}\frac{\partial^{\alpha}(f)}{\partial x^{\alpha}}\bd(x)^{\alpha+\beta}a_{\beta}\notag\\
&=\sum_{1\leq|\gamma|\leq n}\Big(\sum_{\substack{0\leq|\beta|\leq n-1 \\ \beta\leq\gamma, \beta\neq\gamma}}
\frac{1}{(\gamma-\beta)!}\frac{\partial^{\gamma-\beta}(f)}{\partial x^{\gamma-\beta}}a_{\beta}\Big)\bd(x)^{\gamma}
\in\Omn\otimes_A\q.\notag
\end{align}
In addition, since $v\neq0$, we have $a_{\beta_0}\neq0$ (in $R/\q$) for some $\beta_0$, i.e, $a_{\beta_0}\notin\q$.

Now let $g=e/h\in\idc'\br\setminus\Rq$ (see lemma \ref{ass}, $(ii)$). Then 
$$\sum_{0\leq|\beta|\leq n-1}ga_{\beta}\ff\in 
g\q\big(\Omn\otimes_A R\big)\subset\idc\br\big(\Omn\otimes_A R\big)\subset \Omn\otimes_A R.$$
It follows by remark \ref{free omr} that there exist $\lambda_{\alpha}\in R$ such that
\begin{equation}\label{eq3}
\sum_{0\leq|\beta|\leq n-1}ga_{\beta}\ff=\sum_{1\leq|\alpha|\leq n}\lambda_{\alpha}\bd(x)^{\alpha}.
\end{equation}
Combining (\ref{eq3}) with remark \ref{omr quot} we obtain the following equality in $\Omr$:
\begin{equation}\label{eq4}
\Big[h\sum_{\alpha}\lambda_{\alpha}\bd(x)^{\alpha}\Big]=\Big[\sum_{\beta}ea_{\beta}\ff\Big]=0.\notag
\end{equation}
Since $h\neq0$ and, by hypothesis, $\tors(\Omr)=0$, there exist $l_{\beta}\in R$ such that
$\sum_{\alpha}\lambda_{\alpha}\bd(x)^{\alpha}=\sum_{\beta}l_{\beta}\ff$. Substituting in (\ref{eq3}),
$$\sum_{0\leq|\beta|\leq n-1}ea_{\beta}\ff=\sum_{0\leq|\beta|\leq n-1}hl_{\beta}\ff.$$
Thus, $\sum_{\beta}(ea_{\beta}-hl_{\beta})\ff=0$. By proposition \ref{lin indep}, $ea_{\beta}-hl_{\beta}=0$ for each $\beta$.
In particular, for $\beta_0$:
$$g=\frac{e}{h}=\frac{l_{\beta_0}}{a_{\beta_0}}\in\Rq.$$
This is a contradiction, since $g\notin\Rq$. We conclude that $R$ is normal.
\end{proof}

\begin{coro}
$\tors(\Om_R)=0$ if and only if $\tors(\Omr)=0$ for all $n\geq1$.
\end{coro}

When $\dim R=1$, we have a new proof of a special case of the so-called Berger's conjecture 
(generalized to $n\geq1$).

\begin{coro}(cf. \cite{B})
$R$ is regular if and only if $\tors(\Omr)=0$ for every $n\geq1$. 
\end{coro}

\begin{rem}
As we mentioned before, theorem \ref{tors} was also proved by J. Lipman for (local) complete intersections. One key
ingredient in his proof is the fact that $pd(\Om_R)\leq1$. Using our theorem \ref{pd}, it is likely that his proof could also be adapted
to prove theorem \ref{tors} for hypersurfaces and for general $n\geq1$. In addition, if theorem \ref{pd} could be generalized for 
complete intersections, it is plausible that Lipman's proof could also be adapted to this more general case.
\end{rem}

\begin{rem}
Notice that the proof of theorem \ref{tors} uses strongly the Jacobian criterion of order $n$ (proposition \ref{jacobian criterion}). This
proposition do not necessarily holds for complete intersections (see \cite[Remark 2.3]{D}). Thus, Suzuki's proof cannot be directly adapted 
for $n\geq1$ and complete intersections.
\end{rem}


\subsection{Examples}

We conclude with some examples that illustrate some other features of the torsion submodule of the module of differentials of order $n$.

\begin{exam}
Let $B=k[x_1,x_2]/\langle x_1^3-x_2^2 \rangle\cong k[t^2,t^3]$, where $k$ is a field and $char(k)\neq3$. 
According to theorem \ref{tors}, the torsion submodule of the module of differentials of order $n$ of $B$ is non-trivial, for every
$n\geq1$. In this example we exhibit an explicit torsion element of $\Omb$ for each $n$.

Recall that $\Om_B=Bdx_1\oplus Bdx_2/\langle 3t^4dx_1-2t^3dx_2\rangle_B$. Letting $m=3t^3dx_1-2t^2dx_2$, it follows that 
$t^3[m]=[0]$ in $\Om_B$. If $[m]=[0]$ then $m=h(3t^4dx_1-2t^3dx_2)$ for some $h\in B$, which is not possible. Thus, 
$[m]\neq[0]$ and $[m]\in\tors(\Om_B)$.

We claim that the following element gives place to a non-zero torsion element of $\Omb$, for all $n\geq1$:
$$m=3t^3\dnb(x_1)\dnb(x_2)^{n-1}-2t^2\dnb(x_2)^{n}.$$
First notice that, by proposition \ref{omb} and recalling that $\Omb=I_B/I_B^{n+1}$,
\begin{align}
t^3[m]&=[3t^6\dnb(x_1)\dnb(x_2)^{n-1}-2t^5\dnb(x_2)^{n}]\notag\\
&=t^2[3t^4\dnb(x_1)\dnb(x_2)^{n-1}-2t^3\dnb(x_2)^{n}]\notag\\
&=t^2[(\dnb(x_2)^{n-1})(3t^4\dnb(x_1)-2t^3\dnb(x_2))]\notag\\
&=t^2[(\dnb(x_2)^{n-1})(\dnb(f))]\notag\\
&=t^2[F_{(0,n-1)}]=[0].\notag
\end{align}
Now we show that $[m]\neq[0]$. To illustrate the strategy, let us explain the case $n=3$. Letting $f=x_1^3-x_2^2$ and using 
the isomorphism $B\cong k[t^2,t^3]$, the Jacobian matrix of order $3$ of $f$ is
\[Jac_3(f)=
\left( 
\begin{array}{cccccccccc}
&3t^4	&-2t^3	&3t^2	&0		&-1 		&1		&0		&0		&0\\
&0		&0		&3t^4	&-2t^3	&0		&3t^2	&0		&-1		&0\\
&0		&0		&0		&3t^4	&-2t^3	&0		&3t^2	&0		&-1\\
&0		&0		&0		&0		&0		&3t^4	&-2t^3	&0		&0\\
&0		&0		&0		&0		&0		&0		&3t^4	&-2t^3	&0\\
&0		&0		&0		&0		&0		&0		&0		&3t^4	&-2t^3
\end{array} 
\right).
\]
In this case, $m=3t^3d^3_B(x_1)d^3_B(x_2)^2-2t^2d(x_2)^3$. If $[m]=[0]$ then, by proposition \ref{omb}, there 
exist $h_{\beta}\in B$, $0\leq|\beta|\leq 2$ such that $m=\sum_{\beta}h_{\beta}F_{\beta}$. This gives place to the following equation
\begin{equation}\label{ec jac3}
Jac_3(f)^t\cdot(h_{\beta})_{0\leq|\beta|\leq2}=(a_{\alpha})_{1\leq|\alpha|\leq3},
\end{equation}
where $a_{\alpha}=0$ for all $\alpha\notin\{(1,2),(0,3)\}$, $a_{(1,2)}=3t^3$, and $a_{(0,3)}=-2t^2$.
Now let $L$ be the submatrix of $Jac_3(f)$ formed by the columns corresponding to the elements of 
$\{(1,0),(2,0),(1,1),(3,0),(2,1)\}$, i.e., 
\[L=
\left( 
\begin{array}{cccccc}
&3t^4	&3t^2	&0		&1		&0\\
&0		&3t^4	&-2t^3	&3t^2	&0\\
&0		&0		&3t^4	&0		&3t^2\\
&0		&0		&0		&3t^4	&-2t^3\\
&0		&0		&0		&0		&3t^4\\	
&0		&0		&0		&0		&0	
\end{array} 
\right).
\]
Then, by equation (\ref{ec jac3}), 
$$L^t\cdot(h_{\beta})_{\{\beta|0\leq|\beta|\leq2\}}=
(a_{\alpha})_{\{\alpha|1\leq|\alpha|\leq3\}\setminus\{(0,1),(0,2),(1,2),(0,3)\}}=(0).$$
By the shape of $L$, we conclude that $h_{\beta}=0$, for $0\leq|\beta|\leq2$, $\beta\neq(0,2)$. Therefore,
$m=h_{(0,2)}F_{(0,2)}=h_{(0,2)}\big(3t^4d^3_B(x_1)d^3_B(x_2)^2-2t^3d^3_B(x_2)^3\big)$. As in case $n=1$, this is 
not possible. Therefore $[m]\neq[0]$ and $[m]\in\tors(\Omega^{(3)}_{B/k})$.

To show that $[m]\neq[0]$ for every $n\geq1$, we follow the same strategy. If $[m]=[0]$ then there exist $h_{\beta}\in B$,
$0\leq|\beta|\leq n-1$ such that $m=\sum_{\beta}h_{\beta}F_{\beta}$. 
This gives place to the following equation
\begin{equation}\label{ec jacn}
Jac_n(f)^t\cdot(h_{\beta})_{0\leq|\beta|\leq n-1}=(a_{\alpha})_{1\leq|\alpha|\leq n},
\end{equation}
where $a_{\alpha}=0$ for all $\alpha\notin\{(1,n-1),(0,n)\}$, $a_{(1,n-1)}=3t^3$, and $a_{(0,n)}=-2t^2$. 
Now let $L$ be the submatrix of $Jac_n(f)$ formed by the columns corresponding to 
$\{\beta+(1,0)| 0\leq|\beta|\leq n-1,\beta\neq(0,n-1)\}$. Then $L$ can be divided into two blocks: one corresponding to a 
(square) diagonal matrix and a line having only zeroes (see lemma \ref{echelon} and its proof). As before, this implies that 
$h_{\beta}=0$ for all $0\leq|\beta|\leq n-1$, $\beta\neq(0,n-1)$ and so
$$m=h_{(0,n-1)}F_{(0,n-1)}=h_{(0,n-1)}\big(3t^4\dnb(x_1)\dnb(x_2)^{n-1}-2t^3h_{(0,n-1)}\dnb(x_2)^n\big).$$
Since this is not possible, we conclude that $[m]\neq[0]$.
\end{exam}

\begin{rem}
It was suggested by S. Suzuki that the module of differentials of non-normal finitely generated algebras always have torsion.
\end{rem}

\begin{exam}
It was shown by S. Suzuki that theorem \ref{tors} fails for $n=1$ if the hypothesis of complete intersection is removed. More 
precisely, Suzuki exhibited an example of a normal non-complete intersection whose module of differentials has
torsion. The example is the following (see \cite[Example]{S}). 

Let $f_1=x_2^2-x_1x_3$, $f_2=x_3^2-x_2x_4$, $f_3=x_2x_3-x_1x_4$. Denote as $B$ the quotient ring 
$k[x_1,x_2,x_3,x_4]/\langle f_1,f_2,f_3 \rangle$. $B$ is a normal non-complete intersection $k$-algebra. Let 
$m_1=-x_4dx_1+2x_3dx_2-x_2dx_3$.
Then $[m_1]\in\Omega^1_B=\oplus_{\{\alpha\in\N^4||\alpha|=1\}}Bd(x)^{\alpha}/Bdf_1+Bdf_2+Bdf_3$ is a non-zero torsion 
element since $x_1m_1=x_2df_1$ (that $[m_1]\neq[0]$ can be checked by a direct computation).

Now we show that $\tors(\Omega^{(2)}_B)\neq0$. Consider the following element of 
$\oplus_{\{\alpha\in\N^4|1\leq|\alpha|\leq2\}}Bd(x)^{\alpha}$:
$$m_2=-x_4d^2_B(x_1)^2+2x_3d^2_B(x_1)d^2_B(x_2)-x_2d^2_B(x_1)d^2_B(x_3).$$
Then $x_1m_2=x_2F^1_{(1,0,0,0)}$ (see (\ref{efes})).  Thus, by proposition \ref{omb}, $[m_2]\in\tors(\Omega^{(2)}_B)$.
To show that $[m_2]\neq[0]$ we check that $m_2\notin \langle F^i_{\beta}|i\in\{1,2,3\},0\leq|\beta|\leq1 \rangle_B$ using 
$\mathtt{SINGULAR}$ $\mathtt{4}$-$\mathtt{0}$-$\mathtt{2}$ (\cite{DGPS}) and the Jacobian matrix of order 2 of $f_1,f_2,f_3$
(ordered with graded lexicographical order):
\[
\footnotesize
\left(
		\begin{array}{cccccccccccccc}
		-x_3 & 2x_2  & -x_1 & 0 & 0 & 0 & 1 & -1 & 0 & 0 & 0 & 0 & 0 & 0 \\
		0 & 0 & 0 & 0 & -x_3 & 2x_2  & 0 & -x_1 & 0 & 0 & 0 & 0 & 0 & 0 \\
		0 & 0 & 0 & 0 & 0 & -x_3 & 2x_2  & 0 & -x_1 & 0 & 0 & 0 & 0 & 0  \\
		0 & 0 & 0 & 0 & 0 & 0 & 0 & -x_3 & 2x_2  & -x_1 & 0 & 0 & 0 & 0  \\
		0 & 0 & 0 & 0 & 0 & 0 & 0 & 0 & 0 & 0 & -x_3 & 2x_2 & -x_1 & 0 \\
		0 & -x_4  & 2x_3 & -x_2 & 0 & 0 & 0 & 0 & 0 & 1 & 0 & -1 & 0 & 0 \\
		0 & 0 & 0 & 0 & 0 & -x_4 & 0 & 2x_3 & 0 & 0 & -x_2 & 0 & 0 & 0 \\
		0 & 0 & 0 & 0 & 0 & 0 & -x_4 & 0 & 2x_3 & 0 & 0 & -x_2 & 0 & 0 \\
		0 & 0 & 0 & 0 & 0 & 0 & 0 & 0 & -x_4 & 2x_3 & 0 & 0 & -x_2 & 0 \\
		0 & 0 & 0 & 0 & 0 & 0 & 0 & 0 & 0 & 0 & 0 & -x_4 & 2x_3 & -x_2 \\
		-x_4 & x_3 & x_2 & -x_1 & 0 & 0 & 0 & 0 & 1 & 0 & -1 & 0 & 0 & 0 \\
		0 & 0 & 0 & 0 & -x_4 & x_3 & 0 & x_2 & 0 & 0 & -x_1 & 0 & 0 & 0 \\
		0 & 0 & 0 & 0 & 0 & -x_4 & x_3 & 0 & x_2 & 0 & 0 & -x_1 & 0 & 0 \\
		0 & 0 & 0 & 0 & 0 & 0 & 0 & -x_4 & x_3 & x_2 & 0 & 0 & -x_1 & 0 \\
		0 & 0 & 0 & 0 & 0 & 0 & 0 & 0 & 0 & 0 & -x_4 & x_3 & x_2 & -x_1
		\end{array}
\right).
\]
We conclude that $\tors(\Omega^{(2)}_B)\neq0$.
\end{exam}

\begin{rem}
It remains open to check if
$$m_n=-x_4d^n_B(x_1)^n+2x_3d^n_B(x_1)^{n-1}d^n_B(x_2)-x_2d^n_B(x_1)^{n-1}d^n_B(x_3)$$
gives place to a non-zero torsion element of $\Omb$ for $n\geq3$. As in the previous example, $x_1m_n=x_2F^1_{(n-1,0,0,0)}$.
What seems to be difficult to prove is whether $[m_n]\neq[0]$.
\end{rem}

\section*{Acknowledgements}

We are grateful to Luis Nu\~nez Betancourt for several inspiring discussions.

\vspace{.5cm}
{\footnotesize \textsc {P. Barajas, Universidad Aut\'onoma de Zacatecas.} \\
Email: paulvbg@gmail.com}\\
{\footnotesize \textsc {D. Duarte, Universidad Aut\'onoma de Zacatecas-CONACYT.} \\
E-mail: aduarte@uaz.edu.mx}
\end{document}